\documentclass[11pt,twoside,reqno]{amsart}
\usepackage{amsmath}
\usepackage{amsthm}
\usepackage{amsfonts,amssymb}
\usepackage{bbm}
\usepackage{latexsym}
\usepackage{mathrsfs}
\usepackage[all]{xy}
\usepackage{url}
\usepackage[pdftex]{graphicx}
\usepackage{float}
\usepackage{hyperref}
\usepackage{amssymb,yfonts}
\usepackage{fontenc}
\usepackage{lipsum}
 \usepackage{float}
 \usepackage{tabularx}

\usepackage{mathdots}

\usepackage[boxsize=23pt]{ytableau}

\usepackage[utf8]{inputenc}
\usepackage[english]{babel}
\usepackage{lipsum}
\usepackage{amsthm}
\usepackage{todonotes}
\usepackage[vcentermath]{youngtab}

\theoremstyle{plain}
\newtheorem{lema}{Lemma}[section]
\newtheorem{prop}[lema]{Proposition}
\newtheorem{teo}[lema]{Theorem}

\newtheorem{coro}[lema]{Corollary}

\newtheorem{remark}[lema]{Remark}

\theoremstyle{definition}

\def\a\Si{{\rm{a}\Sigma }}
\def\w\Si{{\rm{w}\Sigma }}

\def\w {{\textrm {w}}}

\def\H{{\mathfrak H}}

\def\min{{\text{min}}}
\def\max{{\text{max}}}
\def\des{{\mathrm{des}}}
\def\Des{{\mathrm{Des}}}

\def\a{\mathrm{a}}

\def\D{\Delta}

\def\ep{\epsilon}

\def \g{\gamma}

\def\i{\imath}

 \newcommand{\Ga}{\Gamma}
 
 \newcommand{\Si}{\Sigma}
 \newcommand{\si}{\sigma}
 
 \newcommand{\al}{\alpha}

\newcommand{\ZS}{\mathbb{Z}_r\wr \Omega_d}
\newcommand{\Z}{\mathbb{Z}}
\newcommand{\sd}{\mathrm{sd}}
\newcommand{\CMS}{\mathrm{CMS}}
\newcommand{\Cr}{\mathcal{C}_r^{I}}
\newcommand{\CrII}{\mathcal{C}_r^{II}}
\def \i{\imath}

\def \les {\preceq_{\imath}}

\def \pf{\mathfrak{p}}

\begin{document}

\title[On multichain subdivisions]{On the $f$-vectors of  $r$-Multichain Subdivisions }

\author[S. Nazir]{Shaheen Nazir}

\address{Department of Mathematics\\
Lahore University of Management Sciences\\
Lahore, Pakistan}

\email{shaheen.nazir@lums.edu.pk}

\begin{abstract}
For a poset $P$ and  an integer $r\geq 1$, let $P_r$ be a collection of all $r$-multichains  in $P$. Corresponding to each strictly increasing map  $\i:[r]\rightarrow [2r]$, there is  an order $\preceq_{\i}$ on $P_r$. Let $\D(G_{\i}(P_r))$ be the clique complex of the graph $G_{\i}$ associated to $P_r$ and $\i$. In a recent paper \cite{NW}, it is shown that $\D(G_{\i}(P_r))$ is a subdivision of $P$ for a class of strictly increasing maps. In this paper, we show that all these subdivisions  have the same $f$-vector. We  give an explicit description  of the transformation matrices from the $f$- and $h$-vectors of $\Delta$ to the $f$- and $h$-vectors of these subdivisions when $P$ is a poset of faces of $\D$.  We study two important subdivisions Cheeger-M\"{u}ller-Schrader's subdivision and the $r$-colored barycentric subdivision which  fall in our class of $r$-multichain subdivisions.

\end{abstract}


\maketitle

\section{Introduction}
Stanley laid a foundation for the enumerative theory of subdivisions of simplicial complexes in \cite{stanley1992subdivisions}. His goal
was to understand the behavior of the $h$-polynomial under iterated subdivisions. In recent years, a lot of studies has
been done continuing the Stanley's program for important classes of subdivisions, e.g., barycentric subdivisions in
\cite{brenti2008f}, edgewise subdivisions in \cite{jochemko2018real}, interval subdivisions in \cite{ANf}, antiprism subdivisions in \cite{athanasiadis2022combinatorics} and uniform subdivisions in \cite{Auniform}. All this enumerative study  began with the work of Brenti and Welker \cite{brenti2008f} on barycentric subdivisions. They  studied the transformation matrix of the $h$-vector of a simplicial complex under the barycentric subdivision. They  proved that the $h$-polynomial of the barycentric subdivision of a simplicial complex with non-negative $h$-vector is real-rooted. Recently, Athansiadis in \cite{Auniform} investigated the entries of the transformation matrix of the $h$-vector of a simplicial complex under the $r$-colored barycentric subdivision. He described them in terms of $r$-colored Eulerian numbers. He  also showed that the $h$-polynomial of the $r$-colored barycentric subdivision of a simplicial complex with non-negative $h$-vector is real-rooted.

Let $P$ be a poset with order relation $\leq$. For a non-negative integer $r$,  an $r$-multichain $\mathfrak{p} : p_1 \leq \cdots \leq p_r$ in $P$ is a monotonically increasing sequence of elements in $P$ of length $r$. We consider
the set $P_r$ of all $r$-multichains in $P$. If $r = 1$ then $P_r = P$ and the order complex $\Delta(P)$ of all linearly ordered subsets of $P$ together with its
geometric realization are well studied geometric and topological objects. They have been shown to encode crucial information about $P$
and have important applications in combinatorics and many other
fields in mathematics (see e.g. \cite{wachs2006poset}). For every strictly monotone map
$\i : [r]  \rightarrow [2r]$, define a binary relation
$\preceq_i$ on $P_r$.
 For $\mathfrak{p} : p_1 \leq \cdots \leq p_r$ and
$\mathfrak{q} :q_1 \leq \cdots \leq q_r$ we set:

$$\mathfrak{p}\les \mathfrak{q} :\iff
                                         \begin{array}{ll}
                                          p_t\geq q_s , & \hbox{for $s\leq \imath(t)-t$;} \\
                                          p_t\leq q_s, & \hbox{for $s>\imath(t)-t$.}
                                         \end{array}
.
$$ for $\mathfrak{p},\mathfrak{q}\in P_r$.
 Here for a natural number $n$ we write $[n]$ for $\{1,\ldots, n\}$. Through the undirected graph
$G_\i(P_r) = (P_r,E)$ with edge set
$$E = \big\{ \{\mathfrak{p},\mathfrak{q} \} \subseteq P_r ~:~\mathfrak{p} \preceq_\i \mathfrak{q} \text{ and } \mathfrak{p} \neq \mathfrak{q} \big\}$$
we associate to $P_r$ and $\i$ the clique complex
$\Delta(G_\i(P_r))$ of $G_\i(P_r)$; that is the simplicial complex of all
subsets $A \subseteq P_r$ which form a clique in $G_\i(P_r)$.
\begin{teo}\cite[Theorem 1.1]{NW}\label{NW}
  For $r\geq 2$, the following are equivalent.
  \begin{itemize}
    \item The relation $\preceq_\i$ is reflexive;
    \item The map $\i$ satisfies the condition that $\i(t)\in \{2t-1,2t\}$ for all $1\leq t\leq r$.
    \item The complex $\Delta(G_\i(P_r))$ is a subdivision of $\Delta(P)$.
  \end{itemize}

  \end{teo}
 It is also shown in \cite{NW} that all subdivisions mentioned in  Theorem \ref{NW} are non-isomorphic. It arises a  natural question whether these subdivisions have the same face enumeration or not. We answer this question affirmatively in Theorem \ref{main}.

\begin{teo} \label{main}
Let $\mathcal{I}$ be the collection
of all strictly increasing maps $\i : [r] \rightarrow [2r]$ such that $\i(1) = 1$ and $\les$ is reflexive. Then the $f$-vector of the clique complex $\Delta(G_\i(P_r))$ is the same for all $\i\in \mathcal{I}$.
\end{teo}

\noindent We give explicit formulae for the  transformation matrix of the $f$-vector under these multichain subdivisions of a simplicial complex. It is shown that the entries of the transformation matrix of the $h$-vector of the $r$-multichain subdivisions are given in terms of  the descent numbers of the $r$-colored permutations. On the way, we formulate some interesting recurrence relations between the $r$-colored Eulerian polynomials. Using these relations and  \cite[Theorem 2.3]{savage2015}, we derive  the real-rootedness of the $h$-polynomial of these chain subdivisions(also  given in \cite[Proposition 7.5]{Auniform}).

\noindent We also investigate two special cases of  $r$-multichain subdivisions. We call the clique complex $\D(G_{\i}(P_r))$ an $r$-multichain subdivision
of type I of $\D(P)$ and denote it by $\D(G_{I}(P_r))$ if $\i$ is defined as $\i(t) = 2t-1$ for all $1\leq t\leq r$. For this $\i$, the relation $\preceq_{\i}$ is denoted as $\preceq_{I}$. We call the clique complex $\D(G_{\i}(P_r))$ an $r$-multichain subdivision of type II of $\D(P)$ and denote it by $\D(G_{II}(P_r))$ when $\i$ is defined as $\i(t)= 2t$,  for $t$  even; $\i(t)=2t-1$, for $t$  odd. In this case, the relation $\preceq_{\i}$ is denoted as $\preceq_{II}$.

\begin{figure}[h]
\begin{minipage}{0.45\linewidth}

\centering
\begin{tikzpicture}[scale=0.4]
\filldraw[color=black, fill=gray!30,ultra thick]  (-3,0) -- (9,0) -- (3,12)  -- cycle;
\filldraw[black] (-3,0) circle (8pt) node[anchor=north]{};
\filldraw[black] (9,0) circle (8pt) node[anchor=north]{};
\filldraw[black] (3,12) circle (8pt) node[anchor=north]{};
\filldraw[black] (-1,0) circle (8pt) node[anchor=north]{};
\filldraw[black] (1,0) circle (8pt) node[anchor=north]{};
\filldraw[black] (3,0) circle (8pt) node[anchor=north]{};
\filldraw[black] (5,0) circle (8pt) node[anchor=north]{};
\filldraw[black] (7,0) circle (8pt) node[anchor=north]{};
\draw[black, ultra thick] (-3,0) -- (6,6);
\draw[black, ultra thick] (3,12) -- (3,0);
\draw[black, ultra thick] (9,0) -- (0,6);
\filldraw[black] (6,6) circle (8pt) node[anchor=north]{};
\filldraw[black] (0,6) circle (8pt) node[anchor=west]{};
\filldraw[black] (3,4) circle (8pt) node[anchor=west]{};

\filldraw[black] (-2,2) circle (8pt) node[anchor=west]{};
\filldraw[black] (-1,4) circle (8pt) node[anchor=west]{};
\filldraw[black] (1,8) circle (8pt) node[anchor=west]{};
\filldraw[black] (2,10) circle (8pt) node[anchor=west]{};

\filldraw[black] (8,2) circle (8pt) node[anchor=west]{};
\filldraw[black] (7,4) circle (8pt) node[anchor=west]{};
\filldraw[black] (5,8) circle (8pt) node[anchor=west]{};
\filldraw[black] (4,10) circle (8pt) node[anchor=west]{};

\filldraw[black] (3,1.33) circle (8pt) node[anchor=west]{};
\filldraw[black] (3,2.66) circle (8pt) node[anchor=west]{};
\filldraw[black] (3,6.66) circle (8pt) node[anchor=west]{};
\filldraw[black] (3,9.33) circle (8pt) node[anchor=west]{};

\filldraw[black] (-1,1.33) circle (8pt) node[anchor=west]{};
\filldraw[black] (1,2.66) circle (8pt) node[anchor=west]{};
\filldraw[black] (4,4.66) circle (8pt) node[anchor=west]{};
\filldraw[black] (5,5.33) circle (8pt) node[anchor=west]{};

\filldraw[black] (1,5.33) circle (8pt) node[anchor=west]{};
\filldraw[black] (2,4.66) circle (8pt) node[anchor=west]{};
\filldraw[black] (5,2.66) circle (8pt) node[anchor=west]{};
\filldraw[black] (7,1.33) circle (8pt) node[anchor=west]{};

\draw[black, thick] (-1,0) -- (-1,1.33);
\draw[black,  thick] (1,0) -- (1,2.66);
\draw[black, thick] (-1,1.33) -- (3,1.33);
\draw[black, thick] (1,2.66) -- (3,2.66);
\filldraw[black] (1,1.33) circle (8pt) node[anchor=west]{};
\draw[black, thin] (1,0) -- (-1,1.33);
\draw[black,  thin] (1,2.66) -- (3,1.33);
\draw[black, thin] (1,0) -- (3,1.33);

\draw[black, thick] (7,0) -- (7,1.33);
\draw[black,  thick] (5,0) -- (5,2.66);
\draw[black, thick] (3,1.33) -- (7,1.33);
\draw[black, thick] (3,2.66) -- (5,2.66);
\filldraw[black] (5,1.33) circle (8pt) node[anchor=west]{};
\draw[black, thin] (5,0) -- (3,1.33);
\draw[black,  thin] (5,0) -- (7,1.33);
\draw[black, thin] (3,1.33) -- (5,2.66);

\draw[black, thick] (-1,1.33) -- (-2,2);
\draw[black,  thick] (1,2.66) -- (-1,4);
\draw[black, thick] (-1,1.33) -- (1,5.33);
\draw[black, thick] (1,2.66) -- (2,4.66);
\filldraw[black] (0,3.33) circle (8pt) node[anchor=west]{};
\draw[black, thin] (1,5.33) --(1,2.66);
\draw[black,  thin] (-1,1.33) --(-1,4);
\draw[black, thin] (-1,4) -- (1,5.33);

\draw[black, thick] (1,5.33) -- (3,9.33);
\draw[black,  thick] (2,4.66) -- (3,6.66);
\draw[black, thick] (2,10) -- (3,9.33);
\draw[black, thick] (1,8) -- (3,6.66);
\filldraw[black] (2,7.33) circle (8pt) node[anchor=west]{};
\draw[black, thin] (3,6.66) --(1,5.33);
\draw[black,  thin] (1,8) --(1,5.33);
\draw[black, thin] (1,8) -- (3,9.33);

\draw[black, thick] (4,10) -- (3,9.33);
\draw[black,  thick] (5,8) -- (3,6.66);
\draw[black, thick] (5,5.33) -- (3,9.33);
\draw[black, thick] (4,4.66) -- (3,6.66);
\filldraw[black] (4,7.33) circle (8pt) node[anchor=west]{};
\draw[black, thin] (3,6.66) --(5,5.33);
\draw[black,  thin] (3,9.33) --(5,8);
\draw[black, thin] (5,8) -- (5,5.33);

\draw[black, thick] (7,1.33) -- (8,2);
\draw[black,  thick] (5,2.66) -- (7,4);
\draw[black, thick] (7,1.33) -- (5,5.33);
\draw[black, thick] (5,2.66) -- (4,4.66);
\filldraw[black] (6,3.33) circle (8pt) node[anchor=west]{};
\draw[black, thin] (7,4) --(7,1.33);
\draw[black,  thin] (5,5.33) --(5,2.66);
\draw[black, thin] (7,4) -- (5,5.33);
\end{tikzpicture}

\caption{$3$-chain subdivision of type I of the order complex of the poset $P=\{1<2<3\}$}

\end{minipage}
\begin{minipage}{0.45\linewidth}
 \centering
\begin{tikzpicture}[scale=0.4]
\filldraw[color=black, fill=gray!30,ultra thick]  (-3,0) -- (9,0) -- (3,12)  -- cycle;
\filldraw[black] (-3,0) circle (8pt) node[anchor=north]{};
\filldraw[black] (9,0) circle (8pt) node[anchor=north]{};
\filldraw[black] (3,12) circle (8pt) node[anchor=north]{};
\filldraw[black] (-1,0) circle (8pt) node[anchor=north]{};
\filldraw[black] (1,0) circle (8pt) node[anchor=north]{};
\filldraw[black] (3,0) circle (8pt) node[anchor=north]{};
\filldraw[black] (5,0) circle (8pt) node[anchor=north]{};
\filldraw[black] (7,0) circle (8pt) node[anchor=north]{};
\draw[black, ultra thick] (-3,0) -- (6,6);
\draw[black, ultra thick] (3,12) -- (3,0);
\draw[black, ultra thick] (9,0) -- (0,6);
\filldraw[black] (6,6) circle (8pt) node[anchor=north]{};
\filldraw[black] (0,6) circle (8pt) node[anchor=west]{};
\filldraw[black] (3,4) circle (8pt) node[anchor=west]{};

\filldraw[black] (-2,2) circle (8pt) node[anchor=west]{};
\filldraw[black] (-1,4) circle (8pt) node[anchor=west]{};
\filldraw[black] (1,8) circle (8pt) node[anchor=west]{};
\filldraw[black] (2,10) circle (8pt) node[anchor=west]{};

\filldraw[black] (8,2) circle (8pt) node[anchor=west]{};
\filldraw[black] (7,4) circle (8pt) node[anchor=west]{};
\filldraw[black] (5,8) circle (8pt) node[anchor=west]{};
\filldraw[black] (4,10) circle (8pt) node[anchor=west]{};

\filldraw[black] (3,1.33) circle (8pt) node[anchor=west]{};
\filldraw[black] (3,2.66) circle (8pt) node[anchor=west]{};
\filldraw[black] (3,6.66) circle (8pt) node[anchor=west]{};
\filldraw[black] (3,9.33) circle (8pt) node[anchor=west]{};

\filldraw[black] (-1,1.33) circle (8pt) node[anchor=west]{};
\filldraw[black] (1,2.66) circle (8pt) node[anchor=west]{};
\filldraw[black] (4,4.66) circle (8pt) node[anchor=west]{};
\filldraw[black] (5,5.33) circle (8pt) node[anchor=west]{};

\filldraw[black] (1,5.33) circle (8pt) node[anchor=west]{};
\filldraw[black] (2,4.66) circle (8pt) node[anchor=west]{};
\filldraw[black] (5,2.66) circle (8pt) node[anchor=west]{};
\filldraw[black] (7,1.33) circle (8pt) node[anchor=west]{};

\draw[black, thick] (-1,0) -- (-1,1.33);
\draw[black,  thick] (1,0) -- (1,2.66);
\draw[black, thick] (-1,1.33) -- (3,1.33);
\draw[black, thick] (1,2.66) -- (3,2.66);
\filldraw[black] (1,1.33) circle (8pt) node[anchor=west]{};
\draw[black, thin] (-1,0) -- (1,1.33);
\draw[black,  thin] (1,1.33) -- (3,2.66);
\draw[black, thin] (1,0) -- (3,1.33);

\draw[black, thick] (7,0) -- (7,1.33);
\draw[black,  thick] (5,0) -- (5,2.66);
\draw[black, thick] (3,1.33) -- (7,1.33);
\draw[black, thick] (3,2.66) -- (5,2.66);
\filldraw[black] (5,1.33) circle (8pt) node[anchor=west]{};
\draw[black, thin] (5,0) -- (3,1.33);
\draw[black,  thin] (5,1.33) -- (3,2.66);
\draw[black, thin] (7,0) -- (5,1.33);

\draw[black, thick] (-1,1.33) -- (-2,2);
\draw[black,  thick] (1,2.66) -- (-1,4);
\draw[black, thick] (-1,1.33) -- (1,5.33);
\draw[black, thick] (1,2.66) -- (2,4.66);
\filldraw[black] (0,3.33) circle (8pt) node[anchor=west]{};
\draw[black, thin] (-2,2) --(0,3.33);
\draw[black,  thin] (0,3.33) --(2,4.66);
\draw[black, thin] (-1,4) -- (1,5.33);

\draw[black, thick] (1,5.33) -- (3,9.33);
\draw[black,  thick] (2,4.66) -- (3,6.66);
\draw[black, thick] (2,10) -- (3,9.33);
\draw[black, thick] (1,8) -- (3,6.66);
\filldraw[black] (2,7.33) circle (8pt) node[anchor=west]{};
\draw[black, thin] (2,4.66) --(2,7.33);
\draw[black,  thin] (2,7.33) --(2,10);
\draw[black, thin] (1,8) -- (1,5.33);

\draw[black, thick] (4,10) -- (3,9.33);
\draw[black,  thick] (5,8) -- (3,6.66);
\draw[black, thick] (5,5.33) -- (3,9.33);
\draw[black, thick] (4,4.66) -- (3,6.66);
\filldraw[black] (4,7.33) circle (8pt) node[anchor=west]{};
\draw[black, thin] (4,4.66) --(4,7.33);
\draw[black,  thin] (4,7.33) --(4,10);
\draw[black, thin] (5,8) -- (5,5.33);

\draw[black, thick] (7,1.33) -- (8,2);
\draw[black,  thick] (5,2.66) -- (7,4);
\draw[black, thick] (7,1.33) -- (5,5.33);
\draw[black, thick] (5,2.66) -- (4,4.66);
\filldraw[black] (6,3.33) circle (8pt) node[anchor=west]{};
\draw[black, thin] (8,2) --(6,3.33);
\draw[black,  thin] (6,3.33) --(4,4.66);
\draw[black, thin] (7,4) -- (5,5.33);
\end{tikzpicture}

\caption{$3$-chain subdivision of type II of the order complex of the poset $P=\{1<2<3\}$}

\end{minipage}
\end{figure}
  The main motivation to study these two chain subdivisions is that it  leads us two important geometric subdivisions. One of them is a generalization of the interval subdivision introduced by Walker \cite{walker1988canonical}. In fact, the interval subdivision is a special case of a subdivision described by Cheeger-M\"{u}ller-Schrader in \cite{cheeger1984curvature} for $N=1$.  The other subdivision is  the $r$-colored barycentric subdivision(the $r$-edgewise subdivision of the barycentric subdivision). We give a combinatorial equivalence of these subdivisions (CMS and $r$-colored barycentric) in terms of the $r$-multichain subdivisions.  These connections also lead us to answer a couple of  questions posed by Mohammadi and Welker in \cite{MW}. \\

The paper is organized as follows. In the second section, we provide some background about simplicial complexes and related key words. We recall  some important subdivisions, e.g., barycentric, $r$-edgewise, $r$-colored barycentric, CMS's subdivisions.   In Section 3, the $r$-colored Eulerian polynomials are defined along with underlined recurrence relations. Furthermore, it is shown that these polynomials are real-rooted.  We give some combinatorial description of the $\gamma$-coefficients of the symmetric $r$-colored Eulerian  polynomials. In  Section 4, we prove the main theorem that the $f$-vector of the clique complex $\Delta(G_\i(P_r))$ of $G_\i(P_r)$ does not depend on $\i$ when the relation $\preceq_{\i}$ is reflexive.  We  also describe  the transformation of the $f$- and $h$-vectors under these chain subdivisions of a simplicial complex and show that every $r$-multichain subdivision of a Cohen-Macaulay simplicial complex has the real-rooted $h$-vector. In the last section, we discuss the connection between the $r$-multichain subdivisions with other well-known subdivisions. In Proposition 5.1, we show that for even values of  $r$, the $r$-multichain subdivision of type I(defined in Section 1)  gives a combinatorial description of the CMS subdivision. In Proposition 5.2, we show that the $r$-multichain subdivision of type II(defined in Section 1) is  isomorphic to the $r$-colored barycentric subdivision.

\section{Preliminaries}
We begin by recalling necessary definitions covering the background.
\subsection{Simplicial Complexes and Face Vectors:} An abstract simplicial complex $\Delta$ on a finite vertex set $V$ is a
collection of subsets of $V,$ such that $\{v\}\in \Delta$ for all
$v\in V$, and if $F\in \Delta$ and $E\subseteq F$, then $E\in
\Delta$. The members of $\Delta$ are known as \textit{faces}. The
dimension $\dim (F)$ of a face $F$ is $|F|-1$. Let $d = \max\{|F| : F \in \D\}$
and define the dimension of $\Delta$ to be $\dim \Delta = d-1$. For each
$F\in \D$, we denote $2^F$ as the simplex with vertex set $F$. One can associate to an abstract simplicial complex $\D$ a topological space $|\D|$ known as geometric realization of $\D$ by taking the convex hull $\mathrm{conv}(F)$ in some Euclidean space  $\mathbb{R}^m$ for  every face $F$ in $\D$. For more details, see \cite[Chapter 16]{toth2017handbook}.\\
The \textit{$f$--polynomial} of a $(d-1)$--dimensional simplicial
complex $\D$ is defined as:
$$f_{\D}(t)=\sum_{F\in \D}t^{\dim (F)+1}=\sum_{i=0}^{d}f_{i-1}t^{i},$$ where $f_i$ is the number of faces of dimension $i$. Note that $\dim \emptyset=-1$, therefore $f_{-1}=1$. The sequence $f(\D)=(f_{-1},f_0,\ldots,f_{d-1})$ is called the \textit{$f$--vector} of $\D$. Define the \textit{$h$--vector}
$h(\D) = (h_0, h_1, \ldots, h_d)$ of $\D$ by the
\textit{$h$--polynomial}:
$$h_{\D}(t):=(1-t)^df_{\D}(t/(1-t))=\sum_{i=0}^{d}h_it^i.$$

We say that two simplicial complexes $\D$ and $\Ga$ on the vertex sets $V$ and $W$ are \textit{isomorphic} if there is a bijection $\theta: V\rightarrow W$ such that $F\in \D$ iff $\theta(F)\in \Ga$.
\subsection{ Subdivisions:}
A \textit{topological subdivision} of a simplicial complex $\D$ is a (geometric) simplicial complex $\D'$ with a map $\theta: \D'\rightarrow \D$ such that,
 for any face $F\in \D$, the following holds: (a) $\D'_F:=\theta^{-1}(2^F)$ is a subcomplex of $\D'$ which is homeomorphic to a ball of dimension $\dim(F)$; (b) the interior of $\D'_F$ is equal to $\theta^{-1}(F)$.  The face $\theta(G)\in \D$ is called  the \textit{carrier} of  $G\in \D'$. The subdivision $\D'$ is called \textit{quasi-geometric } if no face of $\D'$ has the carriers of its vertices contained in a face of $\D$ of smaller dimension. Moreover,
$\D'$ is called \textit{geometric}  if there exists a geometric realization of $\D'$
which geometrically subdivides a geometric realization of  $\D$, in the way prescribed by $\theta$.\\
Clearly, all geometric subdivisions (such as the barycentric, edgewise and chain subdivisions considered in this paper) are quasi-geometric.
For more detail, we refer to \cite{stanley1992subdivisions} and a survey by Athanasiadis \cite{athanasiadis2016survey1}. Moving forward,  we recall some well-known subdivisions.
 \subsubsection{The  barycentric subdivision:}
Let $\{v_1,\ldots, v_n\}$ be an affinely independent set of vectors in $\mathbb{R}^d$. For $\emptyset \neq A\subseteq \{v_1,\ldots, v_n\}$,  let $$b_{A}:=\frac{1}{ |A|}\sum_{v\in A}v$$
be the barycenter of the simplex $\mathrm{conv}(A)$. Then for any chain $\emptyset \neq A_0\subset A_1\subset \cdots\subset A_k$
of subsets of $\{v_1,\ldots, v_n\}$, let $b_{A_0\subset A_1\subset \cdots\subset A_k}:=\mathrm{conv}(b_{A_0},\ldots,b_{A_k})$ be the convex hull.\\
Let $\D_{d-1}$ be a geometric $d-1$-simplex with the vertex set $V=\{e_1,\ldots,e_{d}\}$ of the unit vectors in $\mathbb{R}^d$.
 Then  the set of simplicies $b_{A_0\subset A_1\subset \cdots\subset A_k}$
 for chains $\emptyset \neq A_0\subset A_1\subset \cdots\subset A_k$ of subsets in $V$  defines a subdivision
of $\D_{d-1}$ which is called the barycentric subdivision, denoted by $\mathrm{sd}(\D_{d-1})$, of $\D_{d-1}$. In general, the barycentric subdivision $\sd(\D)$  is obtained from a simplicial complex $\D$ by applying it to every simplex in $\D$.

\subsubsection{The $r$th edgewise subdivision:} Let $\D$ be a  simplicial complex with the vertex set $V_1=\{e_1,e_2,\ldots,e_m\}$ of the unit vectors in $\mathbb{R}^m$. For $u=(u_1,\ldots,u_m)\in \mathbb{Z}^m$, let $\mathrm{Supp}(u):=\{e_i\ :\ u_i\neq 0\}$, and $\imath(u):=(u_1,u_1+u_2,\ldots,u_1+u_2+\cdots+u_m)$. The $r$th edgewise subdivision of $\D$ is the simplicial complex $(\D)^{<r>}$ consisting of subsets $G\subseteq V_r=\{(u_1,\ldots,u_m)\ :\ \sum_{i=1}^{m}u_i=r,\ u_i\geq 0\}$ with $\cup_{u\in G}\mathrm{Supp}(u)\in \D$ and either $\imath(u)-\imath(v)\in \{0,1\}^m$ or  $\imath(u)-\imath(v)\in \{0,-1\}^m$ for all $u,v\in G$. For more details, see in
\cite[Definition 6.1]{brun2005subdivisions} and  \cite{edelsbrunner2000edgewise}.

\subsubsection{The $r$-colored barycentric subdivision:}\label{color}
The $r$-colored barycentric subdivision, denoted by $\sd_r(\D)$ of a simplicial complex $\D$ is the $r$th edgewise subdivision of the barycentric subdivision of $\D$.

\subsubsection{The Cheeger-M\"{u}ller-Schrader's subdivision(\cite{cheeger1984curvature}):} \label{CMS}
 Let $\D_{d-1}$ be the standard simplex of dimension $d-1$ in $\mathbb{R}^d$ with the unit vectors $e_j$ as vertices, then
 $$\D_{d-1}:=\{(t_1,\ldots,t_d)\in \mathbb{R}^d\ :\ \sum_{i=1}^{d}t_i=1 \ \hbox{and}\ t_i\geq 0 \ \hbox{for }\ i=1,2,\ldots, d \}.$$
  For each vertex $e_j$, define a hypercube $C_j$ as:
$$C_j:=\{ (t_1,\ldots,t_d)\in \D_{d-1}: t_j\geq t_i\ \hbox{for all }\ i\}.$$
For $i\neq j$, the opposing faces of $C_j$ are given by the pair of hyperplanes $$H_j^{i,0}=\{ (t_1,\ldots,t_d)\in \D_{d-1}: t_i=0\}$$
and  $$H_j^{i,1}=\{ (t_1,\ldots,t_d)\in \D_{d-1}: t_i=t_j\}.$$
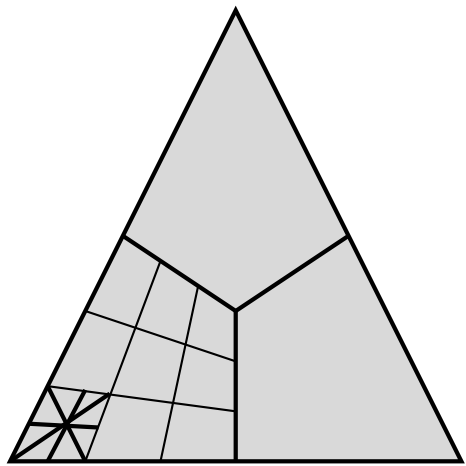
\begin{figure}[h]

\centering
\begin{tikzpicture}[scale=0.5]
\filldraw[color=black, fill=gray!30,ultra thick]  (-3,0) -- (9,0) -- (3,12)  -- cycle;
\draw[black, ultra thick] (3, 0) -- (3,4);
\draw[black, ultra thick] (0,6 ) -- (3,4);
\draw[black, ultra thick] (3,4) -- (6,6);

\draw[black, thick] (-1,0) -- (1,5.33);
\draw[black, thick] (1,0) -- (2,4.66);

\draw[black, thick] (-2,2) -- (3,1.33);
\draw[black,  thick] (-1, 4) -- (3,2.66);
\draw[black, ultra thick] (-3, 0) -- (-0.35,1.8);
\draw[black, ultra thick] (-1, 0) -- (-2,2);
\draw[black, ultra thick] (-2,0) -- (-1,1.9);
\draw[black, ultra thick] (-2.5,1) -- (-0.65,0.9);

\end{tikzpicture}

\caption{CMS subdivision of the $2$-simplex}
\end{figure}

For  a non-negative integer $N$, the hypercube $C_j$'s are further subdivided by  hyperplanes $H_j^{i,k/N}=\{ (t_1,\ldots,t_d)\in \D_{d-1}: t_i=\frac{k}{N} t_j\}$, $0\leq k\leq N$  into $N^{d-1}$ regions, each of which is  a parallelepiped $P$. Now, take the barycentric subdivision of each parallelepiped $P$. The resulting simplicial complex is in fact a subdivision, call it Cheeger-M\"{u}ller-Schrader's Subdivision, denoted as $\CMS(\D_{d-1})$ of the simplex $\D_{d-1}$.
The CMS subdivision $\CMS(\D)$ of a simplicial complex $\D$ is obtained  by applying it to every simplex in $\D$.

\section{The $r$-colored Permutation Group $\Z_r \wr \Omega_d$} Let $d\geq 1$ and $r\geq 0$ be fixed integers. We  present here some  notations and statistics  for the $r$-colored permutation group $\Z_r \wr \Omega_d$, where $\Z_r=\{0,1,\ldots,r-1\}$ is the cyclic group of order $r$ and $\Omega_d$ is the group of usual permutations on $[d]$.  It is the group consisting of
all the bijections $\si$ of the set $$ S:=\{ 1^{(0)},\ldots,d^{(0)}, 1^{(1)},\ldots, d^{(1)},\ldots, 1^{(r-1)},\ldots, d^{(r-1)}\}$$ onto itself with the condition that if $\si({i}^{(s)})= {j}^{(t)}$, then $\si({i}^{(s+1)})= {j}^{(t+1)}$, where the exponents are taken modulo $r$. By the above condition, it is clear that  $\si\in \ZS$ can be fully determined by the first $d$ elements of the set $S$. Therefore,  we may write $\si$ as $(\sigma_1^{\ep_1},\ldots, \sigma_d^{\ep_d})$. The exponent $\ep_i$ can be viewed  as the color assigned to $\si_i$.

 For $\si\in \ZS$, the \textit{descent set} is defined as  $$\textrm{Des}(\si):=\{1\leq i\leq d\ :\ \hbox{either}\ \ep_i>\ep_{i+1}\ \hbox{or}\ \ep_i=\ep_{i+1}\ \hbox{and}\ \si_i>\si_{i+1}\}$$ with the assumption that  $\si_{d+1}:=d+1$ and $\ep_{d+1}:=0$. In particular, $d$ is a descent
of $\si$ if and only if $\si_d$ has nonzero color. The \textit{descent number} of $\si$ is defined  as $\textrm{des}(\si):=|\hbox{Des}(\sigma)|$.\\

Set  $A_{d}:=\{\si\in \ZS :  \ep_1=0\ \}$ and $A_{d,j}:=\{\si\in A_d\ :\  \si_{d}=d+1-j\ \}$.
For $s\in\{0,1,2,\ldots,r-1\}$, set $A_{d,j}^{(s)}:=\{\si\in A_{d,j}\ :  \ \ep_{d}=s\}$, $A_{d}^{(s)}:=\{\si\in A_{d}\ :  \ \ep_{d}=s\}$  and $A_{d}^{(\neq 0)}:=\{\si\in A_{d}\ :  \ \ep_{d}\neq 0\}.$
  The
  \textit{$r$-colored Eulerian polynomials } are defined as follows:
  \begin{equation}\label{B+}
    A^{(s)}_{d,j}(t):=\sum_{\si\in A^{(s)}_{d,j}}t^{\des(\sigma)}=\sum_{m=0}^{d}A^{(s)}(d,j,m)t^m,
  \end{equation} and \begin{equation}\label{B+}
    A^{(s)}_{d}(t):=\sum_{\si\in A^{(s)}_{d}}t^{\des(\sigma)}=\sum_{j=1}^{d}\sum_{m=0}^{d}A^{(s)}(d,j,m)t^m,
  \end{equation}
  where $A^{(s)}(d,j,m)$
 be the  number of elements in $A_{d,j}^{(s)}$ with exactly $m$ descents.\\

  Since $A_{d,j}=\cup_{s=0}^{r-1} A^{(s)}_{d,j}$ and $A_d=\cup_{j=1}^{d}A_{d,j}$ so we have:
   \begin{equation}\label{A_d}
     A_{d,j}(t)=\sum_{s=0}^{r-1} A^{(s)}_{d,j}(t) \ \ \ \ \hbox{and}\  \ \ A_d(t)=\sum_{j=1}^{d}A_{d,j}(t).
   \end{equation}

Some interesting elementary properties and recurrence relations of $A^{(s)}(d+1,k+1,m)$ are given in the following lemma:

\begin{lema}\label{color}
For $0\leq s\leq r-1$ and $0\leq k\leq  d$,
let  $$\H_d^{(s)}{(k)}:=(A^{(s)}(d+1,k+1,0),A^{(s)}(d+1,k+1,1),\ldots, A^{(s)}(d+1,k+1,d)).$$
Then we have the following relations:
\begin{enumerate}
 \item $A^{(0)}(d+1,k+1,m)=A^{(0)}(d+1,d+1-k,d-m)$ and thus $$\H_d^{(0)}(k)=\H_d^{(0)}(d-k)^{\vee},$$ where $(a_0,a_1,\ldots,a_{d-1},a_d)^{\vee}=(a_d,a_{d-1},\ldots, a_1,a_0)$.
 \item  For $s\neq 0$, $A^{(s)}(d+1,k+1,m)=A^{(r-s)}(d+1,d+1-k,d+1-m)$ and thus $$(\H_d^{(s)}(k),0)=(\H_d^{(r-s)}(d-k),0)^{\vee}.$$
    \item
    \begin{eqnarray*}
             A^{(0)}(d+1,k+1,m)&=&\sum_{j=k}^{d-1}A^{(0)}(d,{j+1},m)+\sum_{s=1}^{r-1}\sum_{j=0}^{d-1}A^{(s)}(d,j+1,m) \\
             && +\sum_{j=0}^{k-1}A^{(0)}(d,j+1,m-1).
          \end{eqnarray*}
           Thus, we have:
    $$\H^{(0)}_d({k})=\sum_{j=k}^{d-1}(\H^{(0)}_{d-1}{(j)},0)+\sum_{s=1}^{r-1}\sum_{j=0}^{d-1}(\H^{(s)}_{d-1}(j,0)+\sum_{j=0}^{k-1}(0,\H^{(0)}_{d-1}{(j)}),$$
    with  $\H_0^{(0)}(0)=(1)$ and $\H_0^{(s)}{(0)}=(0)$.
    \item For $s\neq 0$,
     \begin{eqnarray*}
       A^{(s)}(d+1,k+1,m)&=&  \sum_{j=k}^{d-1}A^{(s)}(d,j+1,m)+\sum_{l=1}^{s-1}\sum_{j=0}^{d-1}A^{(l)}(d,j+1,m)\\
       && +\sum_{j=0}^{d-1}A^{(0)}(d,j+1,m-1)+\sum_{j=0}^{k-1}A^{(s)}(d,j+1,m-1)\\
       && +\sum_{l=s+1}^{r-1}\sum_{j=0}^{d-1}A^{(l)}(d,j+1,m-1).
     \end{eqnarray*}
       Thus,  we have
      \begin{eqnarray*}
     \H_d^{(s)}(k)&=&\sum_{j=k}^{d-1}(\H_{d-1}^{(s)}{(j)},0)+\sum_{l=1}^{s-1}\sum_{j=0}^{d-1}(\H_{d-1}^{(l)}{(j)},0)+\sum_{j=0}^{d-1}(0,\H_{d-1}^{(0)}{(j)})\\
    && +\sum_{j=0}^{k-1}(0,\H_{d-1}^{(s)}{{(j)}})+\sum_{l=s+1}^{r-1}\sum_{j=0}^{d-1}(0, \H_{d-1}^{(l)}{(j)}).
    \end{eqnarray*}
\end{enumerate}
\end{lema}
\begin{proof}
  There is a  bijection $\si=(\si_1^{\ep_1},\ldots,\si_{d+1}^{\ep_{d+1}}) \mapsto
\bar{\si}=(\bar{\si}_1^{\bar{\ep_1}},\dots,\bar{\si}_{d+1}^{\bar{\ep_{d+1}}})$ between the set enumerated by the given two numbers,
where
$\bar{\si}_i:=d+1-\si_i$ and
$$\bar{\ep}_i:=\left\{
               \begin{array}{ll}
                 \ep_i, & \hbox{$\ep_i=0$;} \\
                 r-\ep_i, & \hbox{$\ep_i\neq0$.}
               \end{array}
             \right .$$

For $1\leq i\leq d+1$, we have the following four possible cases:
\begin{itemize}
   \item $\ep_i>\ep_{i+1}=0$
   \item $\ep_i>\ep_{i+1}>0$
  \item $\ep_i=\ep_{i+1}=0$ and $\si_i>\si_{i+1}$
  \item $\ep_i=\ep_{i+1}\neq 0$ and  $\si_i>\si_{i+1}$
\end{itemize}
In the first case, $i\in \Des(\si) $ if and only if $i\in \Des(\bar{\si}) $ and in other three cases, we have $i\in \Des(\si) $ if and only if $i\notin \Des(\bar{\si}) $.

(1) In this case, it is clear that  $d+1$ is not a descent of $\si$ and $\bar{\si}$. Thus, $\des(\si)+\des(\bar{\si})=d$ gives the required assertion.

(2) In this case, $d+1$ is always a descent of $\si$ and $\bar{\si}$. Therefore,the required assertion follows from the relation $\des(\si)+\des(\bar{\si})=d+1$.

  (3)  The recursion formula  follows from the effect of removing $\sigma_{d+1}=d+1-k$ from the colored permutation $\si$ in $A_{d+1,k+1}$ with $\des (\si)=r$.

  (4) The proof is similar  as of the assertion (3).
\end{proof}

\begin{coro}\label{poly}
  For $0\leq s\leq r-1$ and $0\leq m\leq  d$, we have the following relations:
\begin{enumerate}
 \item The polynomial  $A_d^{(0)}(t)$ is symmetric.
 \item  The polynomial  $A_d^{(\neq 0)}(t)=\sum_{s=1}^{r-1}A_d^{(s)}(t)$ is symmetric.
 \item For $d\geq1$ and $0\leq k\leq d$, we have $$A_{d,k}^{(0)}(t)=t\sum_{j=0}^{k-1}A_{d-1,j}^{(0)}(t)+\sum_{j=k}^{d-1}A_{d-1,j}^{(0)}(t)+\sum_{s=1}^{r-1}\sum_{j=0}^{d-1}A_{d-1,j}^{(s)}(t).$$
   \item For $s\geq 1$ \begin{eqnarray*}
A_{d,{k}}^{(s)}(t)&=t\sum_{j=0}^{d-1}A_{d-1,j}^{(0)}{(t)}+t\sum_{l=r-1}^{s+1}\sum_{j=0}^{d-1}A_{d-1,j}^{(l)}{(t)}+t\sum_{j=0}^{k-1}A_{d-1,j}^{(s)}{(t)}\\
&+\sum_{j=k}^{d-1}A_{d-1,j}^{(s)}{(t)}+\sum^{l=s-1}_{1}\sum_{j=0}^{d-1}A_{d-1,j}^{(l)}{(t)}.
    \end{eqnarray*}
   \end{enumerate}
\end{coro}
\begin{remark}
  The polynomials $A_{d,k}^{(0)}(t)$ and $A_{d,k}^{(s)}(t)$ from Corollary \ref{poly} (3) and (4) satisfy the same recurrence relation given in \cite[Theorem 2.3]{savage2015}. Thus, $$(A_{d,0}^{(0)}(t),\ldots,A_{d,d}^{(0)}(t),A_{d,0}^{(r-1)}(t),\ldots,A_{d,d}^{(r-1)}(t),\ldots, A_{d,0}^{(1)}(t),\ldots,A_{d,d}^{(1)}(t))$$ is an interlacing sequence of polynomials. This also shows that the polynomials $A_d(t),A_d^{(0)}(t)$ and $A_d^{(\neq 0)}(t)$ are real-rooted.

\end{remark}
\subsection*{The $\g$-vector} The $\g$-vector is also an important enumerative invariant of a flag homology sphere. Gal \cite{gal2005real} conjectured that the $\g$-vector is non-negative for a flag homological sphere. The non-negativity of the $\g$-vector implies the Charnay-Davis conjecture.

 It is well-known that a symmetric polynomial  $p(x)$ of degree $n$ can be uniquely written in the form $$p(x)=\sum_{i=0}^{\lfloor\frac{n}{2}\rfloor}\gamma_ix^i(1+x)^{n-2i},$$ for some $\gamma_i.$ The polynomial $p(x)$ is called \textit{$\gamma$-nonnegative} if $\gamma_i\geq 0$ for all $i$ and $\g=(\g_1,\ldots,g_{\lfloor\frac{n}{2}\rfloor})$ is known as $\g$-vector of polynomial $p(x)$. In this subsection, we aim to provide a combinatorial description of $\g$-vectors of symmetric polynomials $A_d^{(0)}(t)$ and $A_d^{(\neq 0)}(t)$ in terms of some statistics of $r$-colored permutations.\\
Let us first recall the
definition of slide. Let $\si_1^{\ep_1} \cdots \si_d^{\ep_d}$  and
consider $\si_0^{\ep_0}\si_1^{\ep_1} \cdots \si_n^{\ep_d}\si_{d+1}^{\ep_{d+1}}$, where $\si_0=\infty$, $\ep_0=0$,
$\si_{d+1}=d+1$ and $\ep_{d+1}=0$. Put asterisks at each end and also between
$\si_i^{\ep_i}$ and $\si_{i+1}^{\ep_{i+1}}$ whenever $\si_i^{\ep_i}<\si_{i+1}^{\ep_{i+1}}$(${\ep_i}<{\ep_{i+1}}$ or if ${\ep_i}={\ep_{i+1}}$, then $\si_i<\si_{i+1}$).  {\em A  slide
is any segment between asterisks of length at least $2$.} In other
words, a  slide of $\si$ is any decreasing run of
$\si_0^{\ep_0}\si_1^{\ep_1} \cdots \si_d^{\ep_d}\si_{d+1}^{\ep_{d+1}}$ of length at least $2$. For example,
for the permutation $3^{(2)}{5}^{(1)}1^{(0)}2^{(2)}4^{(1)}$,
$*\infty^{(0)}*3^{(2)}5^{(1)}1^{(0)}*2^{(2)}4^{(1)}6^{(0)}*$ there are  two slides, namely, $3^{(2)}5^{(1)}1^{(0)}, 2^{(2)}4^{(1)}6^{(0)}$.\\
 The following theorem is a generalization of  \cite[Theorem 5.3]{ANg}.
\begin{teo}\label{g}
  The polynomials $A_d^{(0)}(t)$ and $A_d^{(\neq 0)}(t):=\sum_{s=1}^{r-1}A_d^{(s)}(t)$ are symmetric of degree $d-1$ , so these can be expressed  as:
  $$A_d^{(0)}(t)=\sum_{i=0}^{\lfloor\frac{d-1}{2}\rfloor}a^{(0)}(d,i,i)t^i(1+t)^{d-1-2i}$$ and
  $$A_d^{(\neq 0)}(t)=\sum_{i=0}^{\lfloor\frac{d-1}{2}\rfloor}a^{(\neq 0)}(d,i,i)t^{i}(1+t)^{d-2i},$$ where
  $a^{(0)}(d,i,i)$ is the number of $r$-colored permutation $\si \in A_d^{(0)}$ with $i$ descents and $i+1$ slides; and
  $a^{(\neq 0)}(d,i,i)$ is the number of $r$-colored permutation $\si \in A_d^{(\neq 0)}$ with $i$ descents and $i+1$ slides.\\
In particular, the polynomials $A_d^{(0)}(t)$ and $A_d^{(\neq 0)}(t)$ are $\gamma$-nonnegative.
\end{teo}
To prove the above theorem, we need to define some notations. Let $A^{(0)}(d,k)$ and $A^{(\neq 0)}(d,k)$  represent the number of all $r$-colored permutations of descent $k$ in $A_d^{(0)}$ and $A_d^{(\neq 0)}$ respectively.
 Let $a^{(0)}(d,k,s)$ and $a^{(\neq 0)}(d,k,s)$) denote the number of $r$-colored permutations with $k$ descent and $s+1$ slides in $A_{d}^{(0)}$ and $A_d^{(\neq 0)}$ respectively. It can be observed that every element in $ A_d^{(0)}$ has at least $1$ slide while an element in $A_d^{(\neq 0)}$ has at least 2 slides.
\begin{lema}
  We have the following relations:
$$a^{(0)}(d,k,s)={d-1-2s\choose k-s}a^{(0)}(d,s,s)\ \  \hbox{and}\ \  a^{(\neq 0)}(d,k,s)={d-1-2s\choose k-s}a^{(\neq 0)}(d,s,s).$$
Therefore,
$$A^{(0)}(d,k)=\sum_{s=0}^{k}{d-1-2s\choose k-s}a^{(0)}(d,s,s)\ \  \hbox{and}\ \  A^{(\neq 0)}(d,k)=\sum_{s=0}^{k}{d-1-2s\choose k-s}a^{(\neq 0)}(d,s,s).$$

\end{lema}
\begin{proof}
  Let us prove the relation for $A_d^{(0)}$. Let $\si\in A_d^{(0)}$ with $s$ descent number and $s+1$ slides. Counting $\si_{0}={\infty}^{(0)}$, there are $d + 1$ symbols and $d + 1- 2(s+1)=d-1-2s$ that are not included in the slides. Choose $k-s$ of these $n-1-2s$ elements, move  each chose element $\si_i^{\ep_i}$ to the left if $\ep_i=0$ (to right  if $\ep_i\neq 0$, respectively) into the nearest slide $*\si_j^{\ep_j}\si_{j+l}^{\ep_{j+l}}*$
with $\si_j^{\ep_j} >\si_i^{\ep_i}>\si_{j+l}^{\ep_{j+l}}.$  After moving chosen elements, the resulting permutation $\bar{\si}$ has exactly $k$ descents and $s+1$ slides. Moreover,  $\bar{\si} $ is still in $A_d^{(0)}$. Thus, the first relation holds. The second assertion follows upon summing $a^{(0)}(n,k,s)$ over $0\leq s\leq k.$ For $A_d^{(\neq 0)}$, the proof follows on similar lines.
\end{proof}
The proof of Theorem \ref{g} follows from the above lemma and the relations $A^{(0)}(d,k)=A^{(0)}(d,d-1-k)$ and $A^{(s)}(d,k)=A^{(r-s)}(d,d-1-k)$ derived from Lemma \ref{color}(1) and (2).

\section{The $f$-vector of $r$-multichain Subdivisions}
In this section, we will prove one of the main results of this paper. Let $\mathcal{I}$ be the collection of all strictly increasing maps  $\i:[r]\rightarrow [2r]$ such that $\i(1)=1$ and $\preceq_{\i}$ is reflexive, i.e. $\i(t)\in \{2t,2t-1\}$ for all $t>1$. Let us recall that   $\D(G_{I}(P_r))$ is the $r$-multichain subdivision of type I when $\i(t)=2t-1$ for all $t$ and  $\preceq_{I}$ is the order relation in $P_r$ in this case. We will prove  that $f(\D(G_{\i}(P_r)))=f(\D(G_{I}(P_r)))$ for all $\i\in \mathcal{I}$.

\begin{proof}[Proof of Theorem \ref{main}]
 Let $F_{k}(\D)$ denote the collection of all $k$-dimensional faces of $\D$. It is clear that $F_0(\D(G_{\i}(P_r)))= F_0(\D(G_{I}(P_r)))$ for all $\i\in \mathcal{I}$. For $k\geq 1$, let  $\pf_1\prec _{\i} \cdots \prec_{\i} \pf_{k+1} $ be a $k$-dimensional face in $\D(G_{\i}(P_r))$, where
$\pf_j:\  p_{j,1}\leq p_{j,2}\leq \cdots \leq p_{j,r}$ is  an $r$-multichain in $P_r$ for $j=1,\ldots,k+1$.
One may represent a $k$-dimensional face $\pf_1\prec _{\i} \cdots \prec_{\i} \pf_{k+1} $ as a matrix
$$M=\left(
  \begin{array}{cccc}
    p_{1,1} & p_{2,1} & \cdots & p_{k+1,1} \\
    p_{1,2} & p_{2,2} & \cdots & p_{k+1,2} \\
    \vdots & \vdots &  & \vdots\\
    p_{1,r} & p_{2,r}& \cdots & p_{k+1,r}\\
  \end{array}
\right)$$
 of order $r\times (k+1)$ with monotonically increasing columns and monotonically increasing $t$-th row when $\i(t)=2t-1$; monotonically decreasing $t$-th row when  $\i(t)=2t$.  One can see that $j$-th column of $M$ represents the $r$-multichain $\pf_j$.\\
For  $\i(t)=2t-1$, define $\overline{p_{j,t}}:=p_{j,t}$. For $\i(t)=2t$, let $(x_1,x_2,\ldots,x_m)$ be the arrangement of distinct elements of $t$-th row $p_{1,t}, p_{2,t}, \ldots, p_{k+1,t}$ in strictly decreasing order. Define $\overline{p_{j,t}}:=x_{m-b+1}$ when $p_{j,t}=x_b$ for some $1\leq b\leq m$. For instance, the monotonically decreasing row  $\pf_t:3\leq 2\leq 1\leq 1$  will be changed to the monotonically increasing row $\overline{\pf_t}: 1\leq 2\leq 3\leq 3$.\\
Consider the  matrix
$$\overline{M}=\left(
  \begin{array}{cccc}
   \overline{ p_{1,1}} & \overline{p_{2,1}} & \cdots & \overline{p_{k+1,1}} \\
   \overline{ p_{1,2}} & \overline{p_{2,2}} & \cdots & \overline{p_{k+1,2}} \\
    \vdots & \vdots & & \vdots\\
   \overline{ p_{1,r}} &\overline{ p_{2,r}}& \cdots & \overline{p_{k+1,r}}\\
  \end{array}
\right)$$
of order $r\times (k+1)$. By definition, each row is monotonically increasing and each column is also monotonically increasing. Moreover, columns of $\overline{P}$ are distinct because the matrix $P$ has distinct columns.\\

 Let  $\overline{\pf_j}: \overline{p_{j,1}}\leq \overline{p_{j,2}}\leq \cdots \leq\overline{ p_{j,r}}$ for $j=1,2,\ldots, k+1$.  Thus, the above matrix  $\overline{M}$ gives us a $k$-dimensional face $\overline{\pf_1}\prec_{I}\cdots \prec_{I} \overline{\pf_{k+1}}$  in $\D(G_{I}(P_r))$ by definition of $\prec_{I}$($\i(t)=2t-1,\ \forall\  t$).\\

 For $\i\in \mathcal{{I}}$ and $k\geq 1$, define a map $\mathcal{F}_{\i}: F_k(\D(G_{\i}(P_r)))\rightarrow F_k(\D(G_{I}(P_r)))$ as
 $$ \pf_1\prec _{\i} \cdots \prec_{\i} \pf_{k+1}\mapsto \overline{\pf_1}\prec_{I}\cdots \prec_{I} \overline{\pf_{k+1}}.$$
 We claim that $\mathcal{F}_{\i}$ is bijection.

  \noindent{\bf $\mathcal{F}_{\i}$ is bijective:} Let $\pf: \pf_1\prec_{I} \cdots \prec_{I} \pf_{k+1}$ be a $k$-dimensional face in $\D(G_{I}(P_r))$. Define  $\overline{\pf}: \overline{\pf_1}\prec_{\i}\cdots \prec_{\i} \overline{\pf_{k+1}}, $ where $\overline{p_{j,t}}= p_{j,t}$ if $\i(t)=2t-1$. For $\i(t)=2t$, define  $\overline{p_{j,t}}=x_{m-b+1}$ when $p_{j,t}=x_b$  where $(x_1,\ldots, x_m)$ be the arrangement of distinct  $p_{t,1},\ldots, p_{t,k+1}$ in the decreasing order. It is clear by definition  that $\overline{\pf}$ is the unique  $k$-dimensional face in $\D(G_{\i}(P_r))$ such that $\mathcal{F}_{\i}( \overline{\pf_1}\prec_{\i}\cdots \prec_{\i} \overline{\pf_{k+1}})=\pf_1\prec_{I} \cdots \prec_{I} \pf_{k+1}$. Thus, it shows that $\mathcal{F}_{\i}$ is bijective.
\end{proof}

\subsection{The $f$-vector of $r$-multichain subdivision of type I}
 In this subsection, we consider $P$ the poset of all faces of a simplicial complex $\D$  of dimension $d-1$. We aim to  give an explicit formula for the transformation matrix of the $f$-vector of $\D(G_{\i}(P_r))$ when $\i$ is reflexive. By Theorem \ref{main}, it is enough to study the $f$-vector of one of the subdivisions $\D(G_{\i}(P_r))$ of $P$.  Set $\Cr(\Delta):=\D(G_{I}(P_r))$ and $[A_1,\ldots,A_r]:=A_1\subseteq\cdots\subseteq A_r$ where $A_t$ is a face in $\D$ for all $1\leq t\leq r$. \\
By the definition of $\Cr(\Delta)$, a $k$-dimensional face in
$\Cr(\Delta)$ is  a chain
$$[A_{01},\ldots, A_{0r}] \prec_{I}[A_{11},\ldots, A_{1r}] \prec_{I} \cdots \prec_{I}[A_{k1},\ldots, A_{kr}]$$ of
$r$-multichains of faces in $\D$ of length $k+1$.
 The $f_0({\Cr(\D)})$ is the  number of  $r$-multichains
$[A_1,\ldots,A_r]$, where $A_1\subseteq \cdots\subseteq A_r$ for  $A_1,\ldots A_r \in \Delta\setminus
\{\emptyset\}$.   For a  fixed $A\in \Delta$, the number of all possible $r$-multichains of the form $[A_1\ldots,A_{r-1},A_r=A]$  is
\begin{equation}\label{f_0}
  \sum_{l_{r-1}=1}^{l}\sum_{l_{r-2}=1}^{l_{r-1}}\cdots \sum_{l_1=1}^{l_{2}}{l_{2}\choose l_1}\cdots{l_{r-1} \choose l_{r-2}}{l\choose l_{r-1}},
\end{equation}
  where $l=|A|$ and $l_i=|A_i|$ for $1\leq i\leq r-1$. By applying  binomial theorem successively, we  obtain that  the expression \eqref{f_0} is equal to $r^l-(r-1)^{l}$.

 Since there are $f_{l-1}(\Delta)$ choices for $A$ with $|A|=l$, the number
of all possible  $r$-multichains in $C^r(\Delta)$ will be
 \begin{equation}\label{0}
 f_0({\Cr(\Delta)})=
\sum_{l=0}^d \big(r^l-(r-1)^{l}\big)f_{l-1}({\Delta}).\end{equation}
  To compute $f_k({{\Cr(\Delta)}})$, for $k\ge0$, let us introduce
some notations.

 Let $P_k^{\alpha_{1},\ldots,\al_r}$ denote  the number of  chains of
$r$-multichains of length $k+1$ terminating at some fixed $r$-multichain $[A_1,A_2,\ldots, A_r]=[A_1,A_1\cup A'_2,\ldots,A_{r-1}\cup A'_r]$,
where $A'_i=A_i\setminus A_{i-1}$ and $\alpha_i=|A'_{i}|$ for all $2\leq i\leq r$ and $\al_1=|A_1|$.
 By definition,   $P_0^{\alpha_{1},\ldots,\alpha_r}=1$ and $P_{-1}^{\alpha_{1},\ldots,\alpha_r}=0$  for all $\al_i$.\\
 There are ${\al_{1}\choose k_1}\cdots{\al_{r-1} \choose k_{r-1}}{\al_r\choose k_{r}}$  choices of $r$-multichains of the form $[B_1,A_1\cup B_2,\ldots,A_{r-1}\cup B_r]$ with $|B_i|=k_i$ for all $i=1,\ldots,r$ such that $[B_1,A_1\cup B_2,\ldots,A_{r-1}\cup B_r]\preceq_{I} [A_1,A_2,\ldots,A_r]$, i.e., $B_1\subseteq A_1\subseteq A_1\cup B_2\subseteq \cdots \subseteq A_{r-1}\cup B_r\subseteq A_r$  and   the number of all chains of length $k$ terminating at $[B_1,A_1\cup B_2,\ldots,A_{r-1}\cup B_r]$ is $P_{k-1}^{k_1,k_2,\ldots,k_r}$.


 For fixed $k$ and $\al_1,\ldots,\al_r$, the number $P_k^{\alpha_{1},\ldots,\alpha_r}$ satisfies the following recurrence relation:
\begin{equation}\label{rec}
  P_k^{\alpha_{1},\ldots,\alpha_r}=\sum_{k_r=0}^{\al_r}\sum_{k_{r-1}=0}^{\al_{r-1}}\cdots\sum_{k_1=1}^{\al_{1}} {\al_{1}\choose k_1}\cdots{\al_{r-1} \choose k_{r-1}}{\al_r\choose k_{r}} P_{k-1}^{k_1,k_2,\ldots,k_r}-P_{k-1}^{\al_{1},\ldots,\al_r}.
\end{equation}

 In the next lemma, we have derived an explicit formula for  $P_k^{\al_1,\ldots,\al_r}$ by  induction and binomial theorem.
\begin{lema}\label{P}
For given $\al_i$ and $k\geq 0$, the number $P_k^{\al_{1},\ldots,\al_r}$ is  given as:
\begin{equation}\label{II}
  P_k^{\al_{1},\ldots,\al_r}= \sum_{i=0}^{k}
(-1)^{k-i}{{k}\choose{i}}[(i+1)^{\al_2+\cdots+\al_{r}}\big((2i)^{\al_1}-(2i-1)^{\al_1}\big)].
\end{equation}

\end{lema}
\begin{proof}
   For $k=0$, $P_0^{\al_{1},\ldots,\al_r}=1$ and for $k=1$, we have $P_1^{\al_{1},\ldots,\al_r}=(2^{\al_1}-1)2^{{\al_{2}+\cdots+\al_r}}-1$. Thus  one can easily see that \eqref{II} holds for $k=0,1$. Now, suppose that \eqref{II} is true for $k-1$. Substitute the formula of $P_{k-1}^{\al_1\ldots,\al_r}$ in the recurrence relation \eqref{rec}, we have
$$P_k^{\al_{1},\ldots,\al_r}=\sum_{k_r=0}^{\al_r}\sum_{k_{r-1}=0}^{\al_{r-1}}\cdots\sum_{k_1=1}^{\al_{1}} {\al_{1}\choose k_1}\cdots{\al_{r-1} \choose k_{r-1}}{\al_r\choose k_{r}}$$
$$\sum_{i=0}^{k-1}
(-1)^{k-1-i}{{k-1}\choose{i}}[(i+1)^{k_2+\cdots+k_{r}}\big((2i)^{k_1}-(2i-1)^{k_1}\big)] $$
$$-\sum_{i=0}^{k-1}
(-1)^{k-1-i}{{k-1}\choose{i}}[(i+1)^{\al_2+\cdots+\al_{r}}\big((2i)^{\al_1}-(2i-1)^{\al_1}\big)].$$
 Using the binomial formula $r$ times ( summing over  $k_1, k_2,\ldots,k_r$), we have
$$P_k^{\al_{1},\ldots,\al_r}=\sum_{i=0}^{k-1}
(-1)^{k-1-i}{{k-1}\choose{i}}[(i+2)^{\al_2+\cdots+\al_{r}}\big((2i+1)^{\al_1}-(2i)^{\al_1}\big)]$$
$$-\sum_{i=0}^{k-1}
(-1)^{k-1-i}{{k-1}\choose{i}}[(i+1)^{\al_2+\cdots+\al_{r}}\big((2i)^{\al_1}-(2i-1)^{\al_1}\big)].$$
Now, using the identity ${{k-1}\choose {i} }+{{k-1}\choose {i-1} }={{k}\choose {i} }$ we get the required identity.
 \end{proof}

  There are $f_{l-1}(\Delta)$ choices for $A$ with
$|A| = l$ and for a fixed $A$ we have ${l_2 \choose l_1}\cdots {l_{r} \choose l_{r-1}}$  $r$-multichain $A_1\subseteq \cdots \subseteq A_r$, where $A_r=A$
with  $|A_i| =l_i$ for $i=1,\ldots, r$.
Hence, we have
\begin{equation}\label{1}
  f_k({\Cr(\Delta)})=\sum_{l=0}^{d}\big(\sum_{l_{r-1}=1}^{l_r}\cdots \sum_{l_1=1}^{l_{2}}{l_{2} \choose l_1}\cdots{l_r\choose l_{r-1}}P_k^{l_{1},l_{2}-l_{1},\ldots,l_r-l_{r-1}}\big)f_{l-1}({\Delta}).
  \end{equation}
Using Lemma \ref{P} and the application of binomial theorem,  we have the $f$-vector transformation as follows:
\begin{teo}\label{f} Let $\Delta$ be a $(d-1)$-dimensional simplicial complex. Then
  \begin{equation}\label{2}
  f_k({\Cr(\Delta)})=\sum_{l=0}^{d}\sum_{i=0}^{k} (-1)^{k-i} {{k}\choose{i}}\big[(r+ri)^l-(r+ri-1)^{l}\big]f_{l-1}(\Delta).
  \end{equation}
  for $0\leq k\leq d-1$ and  $f_{-1}({\Cr(\Delta)})=f_{-1}(\Delta)=1$.
\end{teo}

The transformation of the $f$-vector of
$\Delta$ to the $f$-vector of $r$-multichain subdivision $\Cr(\Delta)$(also for $\mathcal{C}_{2N}^{II}(\D)$) is given by the matrix:
 $$\mathcal{F}_d=[f_{l,m}]_{0\leq l,m\leq d},$$  where
 $$f_{0,m}=\left\{
               \begin{array}{ll}
                 1, & \hbox{$m=0$;} \\
                 0, & {m>0.}
               \end{array}
             \right.
$$ and for $1\leq l\leq d$, we have
\begin{equation}\label{matrix}
  f_{l,m}=\sum_{i=0}^{l-1}(-1)^{l-1-i}{l-1 \choose i}[(ri+r)^m-(ri+r-1)^m]
\end{equation}
In the following lemma, we give a recurrence relation for $f_{l,m}$:
\begin{lema}\label{b}
For $1\leq l\leq d-1$ and $1\leq m\leq d$,
  $$
\sum_{j=1}^{m}r^j{m\choose j}f_{l,m-j}=f_{l+1,m}.$$
\end{lema}
\begin{proof}
  Using \eqref{matrix}, we have \\ 
$\sum_{j=1}^{m}r^j{m\choose j}f_{l,m-j}$
\begin{eqnarray*}
 &=&\sum_{j=1}^{m}r^j{m\choose j}\sum_{i=0}^{l-1}(-1)^{l-1-i}{l-1 \choose i}[(ri+r)^{m-j}-(ri+r-1)^{m-j}] \\
   &=& \sum_{i=0}^{l-1}(-1)^{l-1-i}{l-1 \choose i}[(ri+2r)^m-(ri+2r-1)^m-(ri+r)^m+(ri+r-1)^m]
\end{eqnarray*}
The last assertion follows by taking sum over $j$. Now, after re-summing and using the identity ${{k-1}\choose {i} }+{{k-1}\choose {i-1} }={{k}\choose {i} }$, we get the required identity.
\end{proof}
In the next lemma, we  show how the numbers $f_{l,m }$ are related to the $r$-colored Eulerian numbers.
\begin{lema}{\label{par}}
  Let $T_{t,j}$ be the collection of all partition $T=T_1|\cdots|T_t|T_{t+1}$ of rank $t$ of $d+1$ elements ranging from $S$ for which every element ${1,2,\ldots, d+1}$ with exactly one color appears in $T$; $\min\ T_1$ of  color $(0)$ and $\max\ T_{t+1}=d+1-j$. Then
$$|T_{t,j}|=\sum_{m=0}^{d}{d-j \choose d-m}f_{t,m}.$$
\end{lema}
\begin{proof}
  To form such
a partition, we first choose $d-m$ elements  among $\{1,\ldots, d-j\}$ to
put in $T_{t+1}$ along with $d+1-j$. This can be done in ${{d-j} \choose {d-m}}$
ways.  For $t>0$, to form $T_1|\ldots |T_{t}$
we need to create a set partition from the remaining $m$ elements,
and this can be done in $f_{k,m}$ ways. We proceed with  proving this claim by using
induction on $t$. For $t=1$, it is trivial. For $t=2$, to form $T_{1}$, we need to put $m$
elements from $\{ 1,\ldots, d+1\}\setminus T_{2}$   such
that $\min\ T_1$ of color $(0)$. This gives $r^m-(r-1)^m$ choices, which is the same as
$f_{1,m}$. Suppose that the number of such set partitions
$T_1|\ldots|T_{t}$ of $m$ elements from $\{1, \ldots,  d+1\}$  (with $\min \ T_1$ of color $(0)$) is
$f_{t,m}$. Now, to form such set partition $T_1|T_2|\ldots|T_{t+1}$ of  $m$ elements, we first choose $i$ elements from $m$
remaining elements, where  $i>0$. This can be done in $m^i{m\choose
i}$ ways; and the set partition  $T_1|\ldots|T_{t}$ from remaining
$m-i$ elements can be done in $f_{k,m-i}$ ways (by induction
hypothesis). Thus we have $\sum_{i=1}^{l}r^i{m\choose i}f_{t,m-i}$
ways to form the required  set partitions of rank $t+1$ of $m$
elements.  By Lemma \ref{b}, we have $$
\sum_{i=1}^{m}r^i{m\choose i}f_{t,m-i}=f_{t+1,m}$$
which completes the proof.
\end{proof}
\subsection{The $h$-vector Transformation:}

In this subsection, we express the $h$-vector of  an $r$-multichain  subdivision of simplicial complex $\D$ in term of the $h$-vector of the  simplicial complex $\D$. It is known that the entries of the transformation matrix of the $h$-vector of $\mathcal{C}_{2}^{II}(\D)$ are given in terms of $2$-colored Eulerian numbers, see \cite[Theorem 3.1]{ANf}. The following  theorem generalizes that the entries of the transformation matrix of the $h$-vector of $\mathcal{C}_{r}^{II}(\D)$ are given in terms of $r$-colored Eulerian numbers.

\begin{teo}\label{h}
The $h$-vector of $\Cr(\D)$ can be represented as:
 $$h({\Cr(\Delta)})= \mathcal{R}_d h(\Delta),$$
where the entries of  the matrix $\mathcal{R}_d$ are given as:
  $$\mathcal{R}_d=[A^{(0)}(d+1,s+1,t)]_{0\leq s,t\leq d}.$$
Thus, the $h$-vector of $\Cr(\Delta)$ will be
\begin{equation}\label{h--vector}
   h(\Cr(\Delta))=[A^{(0)}(d+1,k+1,m)]_{0\leq k,m\leq d}\ h(\Delta)=\sum_{k=0}^{d}h_k\H_d^{(0)}(k),
   \end{equation}
   where  $$\H_d^{(s)}{(k)}:=(A^{(s)}(d+1,k+1,0),A^{(s)}(d+1,k+1,1),\ldots, A^{(s)}(d+1,k+1,d))$$
\end{teo}
\begin{proof}
  Since each set partition $T = T_1| \ldots |T_{t+1}$ can be mapped to a permutation
$\sigma = \sigma(T)$ by removing bars and writing each block in increasing
order such that  $\sigma_{d+1} = d+1-j$, and  $\sigma_{1}$ of color $(0)$. That is, $\sigma\in A_{d+1,j+1} $ with $\textrm{Des}(\sigma)\subset D$, where $D = D(A) = \{|A_0|, |A_0|+|A_1|, \ldots, |A_0|+|A_1|+\ldots +|A_{r-1}|\}$. Thus, the claim follows from Lemma \ref{par} and  $h({\Cr(\Delta)})=\mathcal{H}_d\mathcal{F}_d\mathcal{H}_d^{-1}h(\D)$, where $\mathcal{H}_d$ is the transformation matrix from the $f$-vector to the $h$-vector.
\end{proof}

Using  \cite[Theorem 2.3]{savage2015} and Theorem \ref{h}, we have the following result.

\begin{coro}
 Let $\D$ be a $(d-1)$-dimensional simplicial complex with non-negative $h$-vector. Then the $h$-vector of $\Cr(\D)$ is real-rooted.
\end{coro}

\section{Combinatorial equivalences of the CMS and $r$-colored barycentric subdivisions}
In this section, it is shown that the $r$-multichain subdivisions of type I and  II are the same as the $r$-colored barycentric subdivision and the CMS subdivision described in \cite{cheeger1984curvature} for $r=2N$ respectively.
\subsection{The $r$-colored barycentric subdivision:}
Assume that $\D$ is the $d-1$-simplex on the vertex set $[d]$. By definition, $\sd_r(\D)$ is the $r$th edgewise subdivision of the simplicial complex $\sd(\D)$. Since the edgewise subdivision depends on the linear ordering on the vertex set $V(\sd(\D)):=\{F\ :\   \emptyset \neq F\subseteq [d] \}$, therefore we need to  fix an  ordering on $V(\sd(\D)$. Define an ordering $\preceq$ on $V(\sd(\D))$ as: $F\preceq G$ if $|F|<|G|$ or ($|F|=|G|$ and $F\leq_{\mathrm{lex}} G$), where $\leq_{\mathrm{lex}}$ is a lexicographic ordering on finite sets.\\
 Let $U_r$ be the vertex set of $\sd_r(\D)$, i.e., a  collection of all ordered(given  by $\preceq$) $m$-tuples $u=(u_F\ :\ F\in V(\sd(\D)))$ in $\mathbb{Z}_{\geq 0}^m$ such that $\sum_{F\in V(\sd(\D))} u_F=r$ and $\mathrm{Supp}(u)\in \sd(\D)$; $m=|V(\sd(\D))|$. If $u\in U_r$ with $\mathrm{Supp}(u)=\{G_1,\ldots,G_k\}$, then by definition of barycentric subdivision, we have $G_1\subset \cdots \subset G_k\subseteq [d]$.

\begin{prop}
  Let $\D$ be a $d-1$-dimensional simplex. Then the $r$-multichain subdivision $\Cr(\D)$ is isomorphic to  the $r$-colored barycentric subdivision $\sd_r(\D)$.
 \end{prop}
 \begin{proof}
 First, we will show that there is a bijection between the vertex sets $U_r$ and $C_r(\D)$.\\
Let $u=(u_F\ :\ \emptyset \neq  F\subseteq [d])\in U_r$ with  $\mathrm{Supp}(u)=\{G_1,\ldots,G_k\}$.  Define a map $\theta: U_r\rightarrow C_r(\D)$ as:
$$\theta(u)=[A_1,\ldots,A_r],$$ where  $$A_i=\left\{
                                           \begin{array}{ll}
                                             G_1, & \hbox{$1\leq i\leq u_{G_1}$;} \\
                                             G_2, & \hbox{$u_{G_1}+1\leq i\leq u_{G_1}+u_{G_2}$;} \\
                                             \vdots & \hbox{$\vdots$}\\
                                             G_k, &\hbox{$\sum_{j=1}^{k-1}u_{G_j}+1\leq i\leq  \sum_{j=1}^{k}u_{G_j}=r$.}
                                           \end{array}
                                         \right.
$$

\noindent
For $A=[A_1,\ldots,A_r]\in C_r(\D)$, set $u_F:=|\{i\ : F=A_i\}|$  for $F\in \{A_1,\ldots,A_r\}$ and $u_F:=0$ for $F\notin \{A_1,\ldots,A_r\}$. Since $\sum_{F\in V(\sd(\D))}u_F=r$, there is a unique $u=(u_F\ :\ F\in V(\sd(\D)))\in U_r$ such that $\theta(u)=A$. This shows that $\theta$ is a bijection.\\
Since both simplicial complexes $\sd_r(\D)$ and $\Cr(\D)$ are flag so it is enough to show that   $F\in \sd_r(\D)$  if and only if  $\theta(F)\in\Cr(\D)$ for any $1$-dimensional face $F$.

 Let $u,v\in U_r$ such that $\{u,v\}$ is a 1-dimensional face in $\sd_r(\D)$ with $\i(u)-\i(v)\in \{0,1\}^m$. Let $\mathrm{Supp}(u)=\{G_1,\ldots, G_k\}$ and $\mathrm{Supp}(v)=\{H_1,\ldots, H_l\}$. Then $$\i(u)_F=\left\{
                                                        \begin{array}{ll}
                                                          0, & \hbox{$F\preceq H_1$ ;} \\
                                                          u_{H_1}+\cdots+u_{H_j}, & \hbox{$H_j\preceq F\prec H_{j+1}$;} \\
                                                          r, & \hbox{$F\succeq H_{k}$.}
                                                        \end{array}
                                                      \right.
$$ and $$\i(v)_F=\left\{
                                                        \begin{array}{ll}
                                                          0, & \hbox{$F\preceq G_1$ ;} \\
                                                          v_{G_1}+\cdots+v_{G_j}, & \hbox{$G_j\preceq F\prec G_{j+1}$;} \\
                                                          r, & \hbox{$F\succeq G_{l}$.}
                                                        \end{array}
                                                      \right.
$$
Since $\mathrm{Supp}(u)\cup\mathrm{Supp}(v)$ is a face(a chain of $H$'s and $G$'s) in $\sd(\D)$, therefore we must have $H_1\subseteq G_1$ by the assumption that $(\i(u)-\i(v))_{H_1}=0$ or $1$. If $H_2\subset G_1$, then $\i(u)_{H_2}=u_{H_1}+u_{H_2}>1$ and $\i(v)_{H_2}=0$ which contradicts to the supposition that $(\i(u)-\i(v))_{H_2}=0 $ or $1$. Therefore, we must have $G_1\subseteq H_2$. Continuing with this argument, we get consequently  that $H_1\subseteq G_1\subseteq H_2\subseteq \cdots$. This shows that $\theta(u)\prec_{I} \theta(v)$, i.e., $\{\theta(u),\theta(v)\}$ is 1-dimensional face in $\Cr(\D)$.\\
Now, let $A=[A_1,\ldots,A_r]$ and $B=[B_1,\ldots, B_r]$ in $C_r(\D)$ such that $A\prec_{I} B$. Let $u=\theta^{-1}(A)$ and $v=\theta^{-1}(B)$. It implies that $\mathrm{Supp}(u)=\{A_{i_1},\ldots, A_{i_k}\}$ and $\mathrm{Supp}(v)=\{B_{j_1},\ldots, B_{j_l}\}$ and $A_{i_1}\subseteq B_{j_1}\subseteq \cdots$. Therefore, by definition of $u$'s and $v$'s, we have $(\i(u)-\i(v))_{F}=0$ or $1$ for all $F\in V(\sd(\D))$. Thus, $\{u,v\}$ is a 1-dimensional face in $\sd_r(\D)$.
 \end{proof}

\subsection{The CMS subdivision:} We begin with fixing a labeling of CMS subdivided simplicial complex through its simplicies constructively.  Continuing the description in Subsection \ref{CMS}, we assert that the  vertices appearing  in $C_j$ after choosing hyperplanes are resultant of  the intersection of hyperplanes $\cap_{i\neq j}H_j^{i,k_i}$, $0\leq k_i\leq N$. Therefore, the coordinates of these vertices are:
 $$x_i=\left\{
                        \begin{array}{ll}
                          \frac{N}{M}, & \hbox{$i= j$;} \\
                          \frac{k_i}{M}, & \hbox{$i\neq j$.}
                        \end{array}
                      \right.
$$ where $M=N+\sum_{l\neq j}k_l.$\\
 Let us label these vertices by the $d$-tuple $(k_1,\ldots,k_{j-1},N,k_{j+1},\ldots,k_d)$ for $0\leq k_i\leq N$.

Under this labeling, every $m$-dimensional face $F$ of some parallelepiped $P$   in $C_j$ is determined by $2^{m}$ vertices $$\{(l_1,\ldots,l_{j-1},N,l_{j+1},\ldots,l_d)\ :\ l_i=k_i\  \hbox{or}\ k_i+1\ \hbox{with}\  |\{i\ :\ l_i\neq k_i\}|\leq m \}$$ where $k_i=\min\{v_i\ :\ \hbox{$v=(v_1,\ldots,v_d)$ is a vertex of the face $F$} \}.$
  For example, two  vertices $(k_1,\ldots,k_{j-1},N,k_{j+1},\ldots,k_d)$ and  $(l_1,\ldots,l_{j-1},N,l_{j+1},\ldots,l_d)$ in $C_j$ form an edge of a face $F$ of some parallelepiped $P$   in $C_j$   if and only if $|k_{i_0}-l_{i_0}|=1$ for some unique $i_0\neq j$ and  $|k_i-l_i|=0$ for all $i\neq i_0$. \\
 The barycenter $b_F$ of an $m$-dimensional face $F$  of some parallelepiped $P$  in $C_j$ can be  labeled by $(l_1,\ldots,l_{j-1},N,l_{j+1},\ldots,l_d)$, where $$l_i=\left\{
                                                                                  \begin{array}{ll}
                                                                                   k_i, & \hbox{$i$th coordinate remains fixed for all vertices in $F$;} \\
                                                                                   k_i+\frac{1}{2}, & \hbox{otherwise.}
                                                                                  \end{array}
                                                                                \right.
$$ where $k_i=\min\{v_i\ :\ \hbox{$v=(v_1,\ldots,v_d)$ is a vertex of the face $F$} \}.$ It can be observed that the number of non-integers in the coordinate of the vertex $b_F$ is the same as the dimension of $F$.
Thus, the vertex set  $V(\CMS(\D))$ of the CMS subdivision can be labelled as
$$V(\CMS(\D))=\{(\frac{k_1}{2},\ldots,\frac{k_d}{2})\ |\ \text{there exists}\  j\ \text{ such that }\ k_j=2N\ \text{ and }\  0\leq k_i\leq 2N \text{ for all } i\}.$$
Here, we include a figure \ref{example} (when $N=1$ and $d=3$ ) to demonstrate the above labelling.
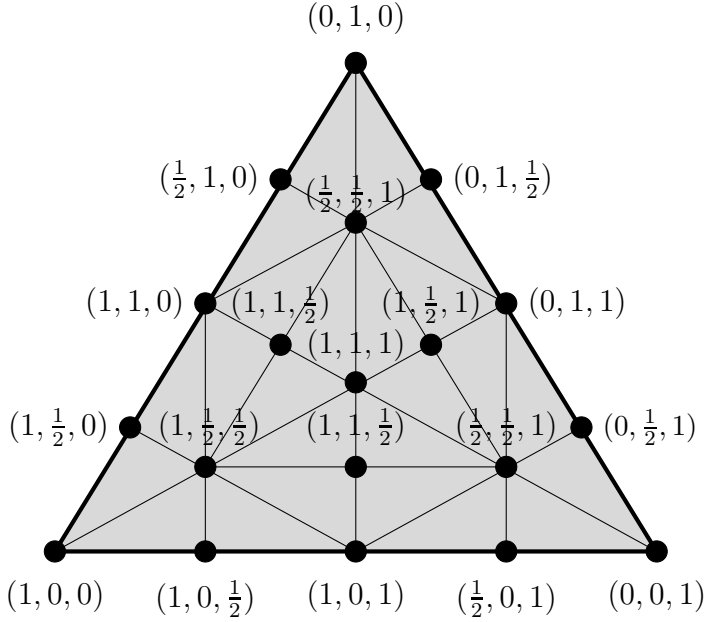
\begin{figure}[h]
\label{example}

\centering
\begin{tikzpicture}[scale=0.5]
\filldraw[color=black, fill=gray!30,ultra thick]  (-3,0) -- (13,0) -- (5,13)  -- cycle;
\filldraw[black] (-3,0) circle (8pt) node[anchor=north]{};
 \fill[black,font=\large]
                    (-3,-0.5) node [below] {$(1,0,0)$};
\filldraw[black] (13,0) circle (8pt) node[anchor=west]{};
\fill[black,font=\large]
                    (13,-0.5) node [below] {$(0,0,1)$};
\filldraw[black] (5,13) circle (8pt) node[anchor=west]{};
\fill[black,font=\large]
                    (5,13.5) node [above] {$(0,1,0)$};
\filldraw[black] (1,0) circle (8pt) node[anchor=west]{};
\fill[black,font=\large]
                    (1,-0.5) node [below] {$(1,0,\frac{1}{2})$};
\filldraw[black] (5,0) circle (8pt) node[anchor=west]{};
\fill[black,font=\large]
                    (5,-0.5) node [below] {$(1,0,1)$};
\filldraw[black] (9,0) circle (8pt) node[anchor=west]{};
\fill[black,font=\large]
                    (9,-0.5) node [below] {$(\frac{1}{2},0,1)$};
\filldraw[black] (-1,3.3) circle (8pt) node[anchor=west]{};
\fill[black,font=\large]
                    (-1.3,3.3) node [left] {$(1,\frac{1}{2},0)$};
\filldraw[black] (1,6.6) circle (8pt) node[anchor=west]{};
\fill[black,font=\large]
                    (0.7,6.6) node [left] {$(1,1,0)$};
\filldraw[black] (3,9.9) circle (8pt) node[anchor=west]{};
\fill[black,font=\large]
                    (2.7,9.9) node [left] {$(\frac{1}{2},1,0)$};
\filldraw[black] (11,3.3) circle (8pt) node[anchor=west]{};
\fill[black]
                    (11.3,3.3) node [right] {$(0,\frac{1}{2},1)$};
\filldraw[black] (9,6.6) circle (8pt) node[anchor=west]{};
\fill[black,font=\large]
                    (9.3,6.6) node [right] {$(0,1,1)$};
\filldraw[black] (7,9.9) circle (8pt) node[anchor=west]{};
\fill[black,font=\large]
                    (7.3,9.9) node [right] {$(0,1,\frac{1}{2})$};
\filldraw[black] (5,4.5) circle (8pt) node[anchor=west]{};
\fill[black,font=\large]
                    (5,4.8) node [above] {$(1,1,1)$};
\filldraw[black] (1,2.25) circle (8pt) node[anchor=west]{};
\fill[black,font=\large]
                    (1.1,2.55) node [above] {$(1,\frac{1}{2},\frac{1}{2})$};
\filldraw[black] (9,2.25) circle (8pt) node[anchor=west]{};
\fill[black,font=\large]
                    (9,2.55) node [above] {$(\frac{1}{2},\frac{1}{2},1)$};
\filldraw[black] (5,2.25) circle (8pt) node[anchor=west]{};
\fill[black,font=\large]
                    (5,2.55) node [above] {$(1,1,\frac{1}{2})$};
\filldraw[black] (5,8.75) circle (8pt) node[anchor=west]{};
\fill[black,font=\large]
                    (5,8.75) node [above] {$(\frac{1}{2},\frac{1}{2},1)$};

\filldraw[black] (3,5.5) circle (8pt) node[anchor=west]{};
\fill[black,font=\large]
                    (3,5.8) node [above] {$(1,1,\frac{1}{2})$};
\filldraw[black] (7,5.5) circle (8pt) node[anchor=west]{};
\fill[black,font=\large]
                    (7,5.8) node [above] {$(1,\frac{1}{2},1)$};
\draw[black,  thin] (-3,0) -- (9,6.6);
\draw[black,  thin] (5,13) -- (5,0);
\draw[black, thin] (13,0) -- (1,6.6);
\draw[black, thin] (1,6.6) -- (1,0);
\draw[black,  thin] (9,6.6) -- (9,0);
\draw[black, thin] (-1,3.3) -- (5,0);
\draw[black, thin] (11,3.3) -- (5,0);
\draw[black,  thin] (1,2.25) -- (9,2.25);
\draw[black,  thin] (1,2.25) -- (5,8.75);
\draw[black, thin] (9,2.25) -- (5,8.75);
\draw[black,  thin] (3,9.9) -- (5,8.75);
\draw[black,  thin] (7,9.9) -- (5,8.75);
\draw[black, thin] (1,6.6) -- (5,8.75);
\draw[black,  thin] (9,6.6) -- (5,8.75);

\end{tikzpicture}

\caption{CMS subdivision of the $2$-simplex when $N=1$}
\end{figure}

Let $b_{F_0,\ldots,F_m}$ be an $m$-dimensional simplex in $\CMS(\D)$, where $F_0\subset F_1\subset \cdots\subset F_m$ is an increasing sequence of faces of some parallelepiped $P$ in $C_j$.  Then it is determined by the set of  $m+1$  vertices $\{b_{F_0},\ldots,b_{F_m}\}$ which satisfies  $b_{F_i}=b_{F_0}$ or $b_{F_0}+\frac{1}{2}$ for all $1\leq i\leq d$. Since the number of non-integral coordinates in $F$ is the same as the dimension of $F$, therefore the number of non-integral coordinates in $F_i$ is less or equal  to the number of non-integral coordinates in $F_j$ and the number of integral coordinates in $F_i$ is greater or equal  to the number of integral coordinates in $F_j$ for all $1\leq i<j\leq m$.

\begin{prop}
Let $\D$ be a simplex of dimension $d-1$. Then for $r=2N$, the chain subdivision $\CrII(\D)$  is isomorphic to  the CMS subdivision.
\end{prop}

\begin{proof} Here, we denote $[A_r,\ldots,A_1]$ by an $r$-multichain $A_r\subseteq \cdots \subseteq A_1$.
Assume that $\D$ is a $d-1$-simplex on the vertex set $[d]$. Define a bijection $\varphi$ between the vertex sets $C^{2N}(\D)$ and $V(\CMS(\D))$ as:  $$v=(\frac{k_1}{2},\ldots,\frac{k_d}{2})\mapsto \varphi(v)=[A_{2N},A_{2N-1},\ldots,A_1],$$
  where $A_{2N}=\{i\ :\ k_i=2N\}$ and for $1\leq l<2N$, $A_l=\{i\ :\ k_i=l\}\cup A_{l+1}$.
Since  for each vertex $v \in V(\CMS(\D))$, there is some $j$ such that $v_j=2N$, therefore $j\in A_{2N}$, hence $A_{2N}$ is non-empty. Moreover, $A_{2N}\subseteq \cdots \subseteq A_1\subseteq [d]$. Thus, $[A_{2N},A_{2N-1},\ldots,A_1]$ is the unique element of $C^{2N}(\D)$ associated to a given vertex $v$ in $\CMS(\D)$. Therefore,  $\varphi$ is well-defined.\\
  To show the subjectivity of $\varphi$, let  $[A_{2N},A_{2N-1},\ldots,A_1]$ be a vertex  in $C^{2N}(\D)$, where $ \emptyset\neq A_{2N} \subseteq A_{2N-1}\subseteq\cdots\subseteq A_1$ is a chain of subsets of $[d]$. For each $l\in [d]$,   let $v_l=|\{i\ :\ l\in A_i\}|$,  then $0\leq k_l\leq 2N$.  Since $A_{2N}$ is non-empty therefore, there is an index $j\in [d]$ such that $k_j=2N$.  Thus, this gives us a unique vertex $v=(\frac{v_1}{2},\ldots,\frac{v_d}{2})$ in $V(\CMS(\D))$ and $\varphi(v)=[A_{2N},A_{2N-1},\ldots,A_1]$, since $|\{i\ :\ v_i\geq l\}|=|\{i\ :\ i\in A_{l}\}|=v _l$ for $1\leq l\leq d$.  This shows that $\varphi$ is a bijection.\\
Since both simplicial complexes $\CMS(\D)$ and $\mathcal{C}_{2N}^{II}(\D)$  are flag so it is enough to show that   $\si\in \sd_r(\D)$  iff  $\theta(\si)\in\Cr(\D)$ for any $1$-dimensional simplex $\si$.
 Let $\si$ be a $1$-dimensional simplex in $\CMS(\D)$ with vertices $\{b_{F_0},b_{F_1}\}$, where $F_0\subset F_1$ is a strictly increasing sequence of faces of some parallelepiped $P$ in $C_j$ and $b_{F_i}$ is the barycenter of the face $F_i$.  It can be noted that  \\
 $\{i: \hbox{the $i$th coordinate remains fixed for all vertices in $F_1$ }\}$\\
  $\subseteq\{i: \hbox{the $i$th coordinate remains fixed for all vertices in $F_{0}$ }\}$. \\  Therefore, by definition of $\varphi$ and  $b_{F_i}$, it follows that $$\varphi(b_{F_1})_{2N}\subseteq \varphi(b_{F_{0}})_{2N}\subseteq \varphi(b_{F_{0}})_{2N-1}\cdots \subseteq \varphi(b_{F_{0}})_{2}\subseteq \varphi(b_{F_{0}})_{1}\subseteq \varphi(b_{F_{1}})_{1}.$$ Consequently, we have $$\varphi(b_{F_1})\prec_{II} \varphi(b_{F_0})$$ which gives a chain of length $2$ in $\mathcal{C}_{2N}^{II}(\D)$.

 Now, let  $[A^0_{2N},\ldots,A^0_1]\prec_{II}[A^1_{2N},\ldots,A^1_1]$ be a $2$-chain in $C_{2N}(\D)$. This gives $2$ vectors $b_{F_0}=(\frac{k^0_{1}}{2},\ldots,\frac{k^0_d}{2})$ and $b_{F_1}=(\frac{k^1_1}{2},\ldots,\frac{k^1_d}{2})\}$ for some faces $F_0, F_1$. Since $k_l^h=|\{i\ |\ l\in A_i^h\}|$, then by ordering of $A^h_l$, we get $k^{0}_i=k^{1}_i \ \hbox{or}\ k^{1}_i+\frac{1}{2}$. Therefore, we must have $F_1\subseteq F_0$. Thus, these vectors give rise an edge  in $\CMS(\D)$.
\end{proof}

\subsection*{Acknowledgement}
I would like to  thank Professor Volkmar Welker  for helpful discussion on the subject of paper. I also wish to thank Imran Anwar for several useful discussions and suggestions which lead to several improvements. I  am also grateful
to the anonymous referee for his/her comments and suggestions for improving an earlier
version of the paper.

\subsection*{Data availability}
Data sharing is not applicable to this article as no data sets were generated or analyzed.

\bibliographystyle{amsalpha}
\bibliography{References}
\end{document}